\documentclass[11pt]{amsart}

\usepackage{amsmath,amsthm,amscd,euscript,longtable,graphicx}

\def \r{\mathbb R}

\def \q{\mathbb Q}
\def \z{\mathbb Z}

\def\calT{\mathcal T}
\def\LaZ{\Lambda_{\z}}
\def\LaQ{\Lambda_{\q}}

\DeclareMathOperator{\modd}{mod}

\DeclareMathOperator{\isin}{lsin}
\DeclareMathOperator{\icos}{lcos}
\DeclareMathOperator{\itan}{ltan}

 \DeclareMathOperator{\id}{ld}

\DeclareMathOperator{\il}{l\ell} \DeclareMathOperator{\is}{lS}
\DeclareMathOperator{\lcm}{lcm} 
\DeclareMathOperator{\aff}{Aff} \DeclareMathOperator{\SL}{SL} \DeclareMathOperator{\GL}{GL}


\newtheorem{theorem}{Theorem}[section]
\newtheorem{lemma}[theorem]{Lemma}

\newtheorem{proposition}[theorem]{Proposition}
\newtheorem{corollary}[theorem]{Corollary}

\theoremstyle{remark}
\newtheorem{remark}[theorem]{Remark}
\theoremstyle{definition}
\newtheorem{definition}[theorem]{Definition}
\newtheorem{example}[theorem]{Example}

\author[O.~Karpenkov, A.~Pratoussevitch, R.~Sheppard]{Oleg Karpenkov, Anna Pratoussevitch, Rebecca Sheppard}
\address{Department of Mathematical Sciences\\ University of Liverpool\\ Peach Street \\ Liverpool L69~7ZL.}

\title{Circumscribed Circles in Integer Geometry}

\begin{date}  {\today} \end{date}

\begin{document}

\begin{abstract}
Integer geometry on a plane deals with objects whose vertices are points in~$\z^2$.
The congruence relation is provided by all affine transformations preserving the lattice~$\z^2$.
In this paper we study circumscribed circles in integer geometry.
We introduce the notions of integer and rational circumscribed circles of integer sets.
We determine the conditions for a finite integer set to admit an integer circumscribed circle
and describe the spectra of radii for integer and rational circumscribed circles.
\end{abstract}

\maketitle

\tableofcontents

\section*{Introduction}
In this paper we introduce the notion of circumscribed circles in integer geometry
and investigate their properties.

\vspace{2mm}

The integer distance between two points in the lattice~$\z^2$ is defined in terms of the number of lattice points on the segment between them; see Section~\ref{Some Integer Invariants} for more details.
An {\it integer circle\/} is the locus of all lattice points at a fixed integer distance from a given lattice point.
The properties of integer circles differ substantially from the properties of their Euclidean counterparts.
In fact, using the Basel Problem~\cite{ayoub}, it can be shown that the density of a unit integer circle in~$\z^2$ is positive and equal to~$6/\pi^2$ (see also~\cite{HW2008})
Note that the chords of unit integer circles provide a tessellation which is combinatorially equivalent to the Farey tessellation of the hyperbolic plane, while their radial segments correspond to geodesics in the hyperbolic plane (see~\cite{Series2015,MGO2019}). 

\vspace{2mm}

An {\it integer circumscribed circle\/} of a subset of~$\z^2$ is defined as an integer circle that contains this subset.
While in Euclidean geometry every non-degenerate triangle has a unique circumscribed circle,
this is no longer the case in integer geometry.
In fact, the number of integer circumscribed circles of an integer triangle is infinite.

\vspace{2mm}

This paper aims to provide a comprehensive study of circumscribed circles in integer geometry.
In Theorem~\ref{theorem: finite set and torus} we introduce necessary and sufficient conditions for a finite integer set to admit a circumscribed circle.
As a special case, we discuss the circumscribed circles of integer quadrangles and their Euclidean counterparts. 

\vspace{2mm}

While a finite set might not admit an integer circumscribed circle,
it will have integer dilates that do.
The integer circumscribed circles of the dilates can be interpreted as integer circles with rational centres and radii.
We call the set of all such rational radii the {\it rational spectrum}.
In Theorem~\ref{theorem:rational circumscribed} we describe the structure of rational spectra of finite sets. 

\vspace{2mm}

\noindent
This paper is organized as follows.
In Section~\ref{Basic Notions of Integer Geometry}, we begin with basic definitions of integer geometry and introduce the notion of an integer circle.
In Section~\ref{Integer Circumscribed Circles} we state and prove the conditions under which a finite integer set admits an integer circumscribed circle.
We extend the notion of circumscribed circles to the case of rational radii and rational centers  and describe the spectra of the radii of such circles in Section~\ref{Rational Circumscribed Circles}. 
In Section~\ref{Section-polygons} we discuss integer and rational circumscribed circles for segments, triangles and quadrangles in more detail.
%

\section{Basic Notions of Integer Geometry}
\label{Basic Notions of Integer Geometry}

\subsection{Objects in Integer Geometry}
Consider the plane $\r^2$ with the fixed basis~$(1,0),(0,1)$.
A point in~$\r^2$ is  {\it integer\/} if its coordinates in this basis are integers.
An {\it integer vector\/} is a vector with integer endpoints.
A segment or polygon is {\it integer\/} if all its vertices are integer.
A subset of $\z^2$ is said to be an {\it integer set}.

\vspace{2mm}

An {\it integer affine transformation\/} is an affine transformations that preserves the integer lattice~$\z^2$.
We denote the set of all integer affine transformations by~$\aff(2,\z)$.
Similar to the Euclidean isometries,
$\aff(2,\z)$ contains {\it integer translations},
{\it integer rotations\/}  
and {\it integer symmetries}. 
They correspond to translations by integer vectors,
multiplication by matrices in~$\SL(2,\z)$ 
and multiplication by matrices in $\GL(2,\z){\setminus}\SL(2,\z)$ respectively.

\vspace{1mm}

We say that two integer sets are {\it integer congruent\/} if there exists an integer affine transformation sending one set to another.

\subsection{Some Integer Invariants}
\label{Some Integer Invariants}

Let us recall some basic notions of integer geometry (see~\cite{OK2022}).
The {\it integer length}~$\il(AB)$ of a vector $AB$ in $ \z^2$ is defined as the number of lattice points that the vector passes through, minus one. 
Note that the integer length is given by the greatest common divisor of the differences of coordinates.
The {\it integer distance}~$\id(AB)$ between~$A$ and~$B$ is the integer lengths of~$AB$.

\vspace{2mm}

The {\it integer area}~$\is(ABC)$ of a triangle~$ABC$ is the index of the sub-lattice generated by~$AB$ and~$AC$ in~$\z^2$.
In fact, the integer area is equal to the absolute value of the determinant~$\det(AB,AC)$,
and therefore it is twice the Euclidean area of the triangle~$ABC$.

\subsection{Integer Circles}

We define an \textit{integer circle\/}
with centre~$O\in\z^2$ and radius~$r\in\z$, $r>0$
as the locus of all points~$P$ such that~$\il(OP)=r.$

\begin{proposition}
The intersection of a line with a circle is either empty, or an infinite periodic subset of integer points in the line, or two points.  
\qed
\end{proposition}

The {\it radial} line of an integer circle~$C$ is a line passing through the center of~$C$. 
A radial line of~$C$ intersects~$C$ in two points.
A line $\ell$  is {\it tangent} to the circle $C$ of radius $r$ with center $O$ if $\id(O,\ell)=r$.

\begin{remark}
For every pair of tangent lines of an integer circle there exists an integer isometry of the circle mapping one tangent line to the other. 
\end{remark}

\begin{remark}
Two integer circles of the same radius are integer congruent.
Moreover, one can be mapped to the other by a translation by an integer vector.
\end{remark}

\begin{figure}
\[\includegraphics[scale = 1]{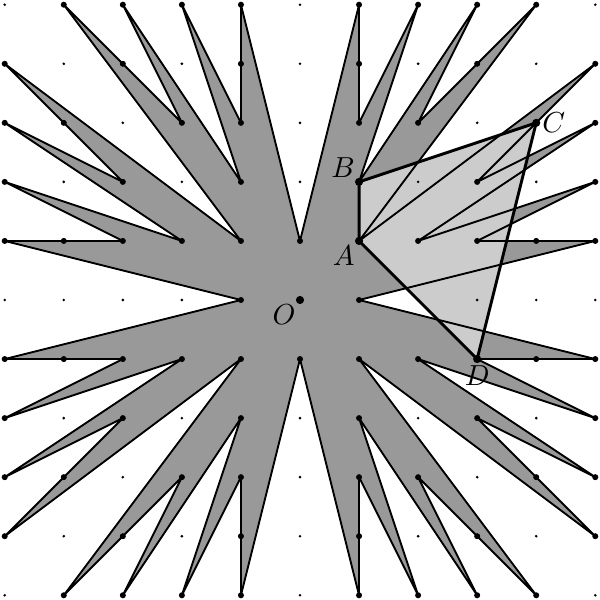}\]
\caption{An integer circle circumscribed about an integer quadrangle.}
\label{fig:starburst}
\end{figure}

\noindent
Figure~\ref{fig:starburst} shows in bold those points of the integer unit circle~$S_0$ centered at the origin~$O$
whose coordinates do not exceed~$5$ in absolute value.
The polygon in Figure~\ref{fig:starburst} is called a {\it Farey starburst} and is obtained by connecting these points by straight segments in the order of increasing argument.
The vertices~$A,B,C,D$ belong to the integer circle~$S_0$,
hence $S_0$ is a circumscribed circle of the quadrangle~$ABCD$.

\begin{remark}
Consider the integer unit circle~$S_0$ centered at the origin~$O$.
Let $\alpha$ be some integer angle and $A$ the point~$(1,0)$.
Then it is possible to find infinitely many points~$B$ in~$S_0$ such that 
the angle $\angle AOB$ is integer congruent to~$\alpha$.
Note the difference with the Euclidean case,
where there are exactly two such points~$B$.
\end{remark}

\subsection{Integer Trigonometry}

Let us discuss basic definitions of integer trigonometry introduced in~\cite{Karpenkov2009,OK2008} (for the multi-dimensional trigonometry see~\cite{BKD2023}).

\begin{definition}
Let $p,q$ be co-prime integers with~$q\ge p>0$.
The {\it integer arctangent of $q/p$} is the angle~$\angle AOB$,
where
$$A=(1,0),\quad O=(0,0),\quad\text{and}\quad B=(p,q).$$
We define {\it integer sine}, {\it integer cosine\/} and {\it integer tangent} as
$$
  \isin \angle AOB=q,\quad
  \icos \angle AOB=p,
  \quad\text{and}\quad
  \itan \angle AOB=q/p.
$$
\end{definition}

Note that any rational angle (that is not contained in a line)
is integer congruent to exactly one integer arctangent.
So the values of integer trigonometric functions form in fact
a complete set of invariants of rational angles up to integer congruence.

\vspace{1mm}

The integer sine has a nice geometric definition:
$$
\isin \angle ABC
=\frac{\is (ABC)}{\il(AB)\il(AC)}
$$
which directly corresponds to the Euclidean formula for the area of a parallelogram in terms of the sine of its angle.
The integer tangent is closely related to the geometry of numbers and their connections to continued fractions~\cite{OK2022}.

\section{Integer Circumscribed Circles}
\label{Integer Circumscribed Circles}
In this section we generalise the notion of a circumscribed circle in the context of integer geometry. 

\vspace{2mm}

\begin{definition}\label{circ-def-1}
An {\it integer circumscribed circle\/} of $S\subset \z^2$ is an integer circle that contains $S$.
\end{definition}

In the Euclidean geometry there exists at most one circumscribed circle for a given set~$S$ with~$|S|>2$.
This is not the case in integer geometry where a set can have several circumscribed circles.
The radius of the circumscribed circle is an important quantity in Euclidean geometry.
A suitable replacement for this quantity in integer geometry is the integer circumscribed spectrum.

\begin{definition}
Let $S$ be an integer set.
The set of all radii of integer circumscribed circles of~$S$ is called the {\it integer circumscribed spectrum\/} of~$S$ and denoted by~$\LaZ(S)$.
\end{definition}

\noindent
Note the following.

\begin{proposition}
\label{max-integer}
Let $a,b\in S$ and let $r$ be the radius of a circumscribed circle of $S$. Then $r$ divides $\id(a,b)$.
\end{proposition}

\begin{proof}
Let $x$ be the centre of the circumscribed circle of~$S$.
Then
\[a-x\equiv b-x\equiv(0,0)\mod r.\]
Hence $a-b\equiv(0,0)\mod r$,
and therefore $r$ divides $\id(a,b)$. 
\end{proof}

\noindent
This proposition implies that the integer spectrum is bounded:

\begin{corollary}\label{lamda z is boundeds}
    The integer spectrum $\LaZ(S)$ of an integer set $S$ that contains at least 2 points is bounded. \qed 
\end{corollary}

The first natural question in the study of integer circumscribed circles is whether $\LaZ(S)$ is empty.
In this section we will introduce a criterion that answers this question for a finite set~$S$ in terms
of projections of~$S$ to integer tori as defined below.
Later in Subsection~\ref{subsection-structure}
we will study the structure of $\LaZ(S)$.

\begin{definition}
For an integer~$m\ge2$, let the $(\modd m)$ {\it integer torus\/} be 
$$
  \mathcal T_m
  =\z^2/\langle(m,0),(0,m)\rangle
  \cong\z/m\z\times\z/m\z.
$$
The \textit{projection} $\pi_m:\z^2\to\calT_m$ is given by $(x,y)\to(x\modd m,y\modd m)$.
\\
We say that two integer points $v_1$ and~$v_2$ in~$\z^2$
are {\it equivalent mod}~$m$ if $\pi_m(v_1)=\pi_n(v_2)$,
denoted by $v_1\equiv v_2\modd m$. 

\end{definition}

In the statement of the main result of this section we use the following terminology.

\begin{definition}
We say that an integer set~$S$
is a {\it covering set\/} of $\calT_m$ if $\pi_m(S)=\calT_m$.
\end{definition}

\begin{definition}
We say that an integer set~$S$ is \textit{tori-transparent\/} if for every integer~$m\ge2$ we have that $S$ is not a covering set of $\calT_m$.
\end{definition}

\begin{remark}
Note that a covering set of an integer torus~$\calT_t$ with~$t\ge2$
must consist of at least $|\calT_t|=t^2\ge4$ points,
hence all integer sets~$S$ with $|S|\le3$ are tori-transparent.
\end{remark}

Now we are ready to write down the existence criterion.

\begin{theorem}
\label{theorem: finite set and torus}
Consider a finite integer set $S\subset\z^2$.
Then the following three statements are equivalent:
\begin{itemize}
\item[(i)] There exists an integer circumscribed circle of~$S$, i.e. 
$$\LaZ(S)\ne\emptyset.$$
\item[(ii)] There exists an integer unit circumscribed circle of~$S$, i.e.
$$1\in \LaZ(S).$$
\item[(iii)] The set $S$ is tori-transparent.
\end{itemize}
\end{theorem}

\noindent
We start the proof with the following four lemmas.

\begin{lemma}
\label{pim-pid}
Let $v_1,v_2\in\z^2$.
Consider two integers $d$ and $m$ such that $d$ is a divisor of~$m$.
Then $\pi_m(v_1)=\pi_m(v_2)$
implies $\pi_d(v_1)=\pi_d(v_2)$.
\end{lemma}

\begin{proof}
If $\pi_m(v_1)=\pi_m(v_2)$, then 
$v_1-v_2 \equiv 0 \modd m$ and hence 
$v_1-v_2 \equiv 0 \modd d$, since $d$ is a divisor of $m$.
Therefore $\pi_d(v_1)=\pi_d(v_2)$.
\end{proof}

\begin{lemma}
\label{lemma: covering divisors}
Let $S$ be any subset of~$\z^2$.
If $S$ is a covering set of~$\calT_m$
then it is a covering set of~$\calT_p$ for any prime divisor~$p$ of~$m$.
\end{lemma}

\begin{proof}
The set~$S$ is a covering set of~$\calT_m$, hence
for each~$v\in\z^2$ there exists some $s\in S$ such that $\pi_m(v)=\pi_m(s)$.
By Lemma~\ref{pim-pid}, $\pi_p(v)=\pi_p(s)$.
Hence, $S$ is a covering set of~$\calT_p$.
\end{proof}


\begin{lemma}\label{prime-not-prime}
For any integer set~$S$ the following statements are equivalent:
\begin{itemize}
\item[(i)] The set $S$ is tori-transparent.
\item[(ii)] The set $S$ is not a covering set of any torus $\calT _p$ for prime $p$.
\end{itemize}
\end{lemma}

\begin{proof}
{\bf (i)$\implies$(ii)}
If the set $S$ is tori-transparent
then $S$ is not a covering set of any torus~$\calT_m$ for integer~$m\ge2$,
hence $S$ is not a covering set of any torus~$\calT_p$ for prime~$p$.

\vspace{2mm}

\noindent
{\bf (ii)$\implies$(i)}
Consider any integer $m\ge 2$ and let $p$ be a prime divisor of~$m$.
By assumption, $S$ is not a covering set of~$\calT_p$. Hence,
by Lemma~\ref{lemma: covering divisors},
$S$ is not a covering set of~$\calT_m$.
\end{proof}

\begin{lemma}\label{finite-cubes}
Consider a finite, tori-transparent integer set $S$.
Then for any finite subset~$M$ of~$\z$
there exists a point~$v\in\z^2$
such that $\pi_m(v)\not\in \pi_m(S)$ for all~$m\in M$.
\end{lemma}

\begin{proof}
Let $\{p_1, \ldots, p_n\}$ be the set of all prime divisors of all elements in $M$.
By Lemma~\ref{lemma: covering divisors}
for every $i=1,\ldots, n$ the set $S$ is not a covering set of~$\calT_{p_i}$.
Hence for every $i=1,\ldots, n$ there exists a point $v_i\in\z^2$ such that  for any $s \in S$ we have
$$
v_i \not \equiv s \modd p_i.
$$
Then by the Chinese Remainder Theorem (applied coordinate-wise) there exists a point $v$ such that 
for every $i=1,\ldots,n$ it holds:
$$
v \equiv v_i \modd p_i.
$$
Hence for every $i=1,\ldots,n$ we have
$$
\pi_{p_i}(v) = \pi_{p_i}(v_i)\notin\pi_{p_i}(S).
$$
Therefore, by Lemma~\ref{lemma: covering divisors}
$\pi_{p_i}(v)\notin \pi_{m}(S)$ for all $m\in M$. 
\end{proof}

\noindent
{\bf Proof of Theorem~\ref{theorem: finite set and torus} (iii)$\implies$(ii).}
The existence of a circumscribed circle and the
property of being a covering set of integer tori $\calT_m$ are invariant under translation by integer vectors. Thus we can assume that the set $S$ is contained in the positive quadrant of $\z^2$. Choose $N$ satisfying the following two conditions:

\begin{itemize}
\item $S$ is completely contained in the box $[1,N]\times [1,N]$;

\item the number of elements in $S$ does not exceed $N$.
\end{itemize}

\vspace{2mm}

Consider $Z=\{1,2\dots,N\}=[1,N]\cap\z$.
By Lemma~\ref{finite-cubes} there exists $(a,b)$ such that $\pi_m(a,b)$ is not in $\pi_m(S)$ for all $m \in Z$. 
 
\vspace{2mm}

Set $\beta=b+N!$.

\vspace{2mm}

Let $p_1, \ldots, p_k$ be all prime numbers in the segment $[N+1,\beta]$. Now note that the size of the set $S$ is $|S|\le N<p_i$. Hence the set of first co-ordinates of points in $S$ has fewer than $p_i$ elements. Therefore, for any $i=1,\ldots, k$ we can choose $c_i$ such that $c_i$ is not equal modulo $p_i$ to the first coordinate of any point in $S$.

\vspace{2mm}

By Chinese Remainder Theorem there exists a solution $\alpha$ of the following system of equations:

$$
\left\{
\begin{array}{l}
\alpha \equiv a \hbox{ mod } N!\\
\alpha \equiv c_i \hbox{ mod } p_i
\end{array} \right.
$$

Then we will show that the point $(\alpha,\beta)$ has the property that $\pi_m(\alpha,\beta) \not \in \pi_m(S)$ for every integer $m$, and therefore $(\alpha,\beta)$ belongs to the unit integer circle with centre at $(x,y)$ for every $(x,y) \in S$.
\begin{itemize}

\item 
If $m\le N$ then $\pi_m(\alpha, \beta)= \pi_m(a,b) \not \in \pi_m(S)$.

\item 
If $m\in[N+1,\beta]$ and $m$ is a prime, say $m = p_i$, then $\alpha \equiv c_i \mod p_i$ and hence is not equal to the first coordinate of any point in $S$ modulo $p_i$ (by the above).
Therefore $\pi_m(\alpha,\beta)\not\in\pi_m(S)$.

\item
If $m\in[N+1,\beta]$ and $m$ is not a prime
then $\pi_m(\alpha,\beta)\not\in\pi_m(S)$
by Lemma~\ref{prime-not-prime} and by the cases considered above.

\item
If $m>\beta$ then the second co-ordinate of any point in $S$ is in the interval $[1,N]$ while $\beta > N! > N$. 
Hence the difference of the second coordinates
is contained in~$[\beta-N,\beta-1]\subset[1,m-1]$
and is therefore not equal to zero modulo~$m$.
Thus $\pi_m(\alpha,\beta)\not\in\pi_m(S)$.\qed
\end{itemize}

\vspace{2mm}

\noindent
{\bf Proof of Theorem~\ref{theorem: finite set and torus} (ii)$\implies$(i).} This is straightforward.
\qed

\vspace{2mm}

\noindent
{\bf Proof of Theorem~\ref{theorem: finite set and torus} (i)$\implies$(iii).}
Assume that there exists a circumscribed circle of~$S$ of some radius~$r$ centered at~$O$. 
Suppose that $S$ is a covering set of~$\calT_m$ for some integer~$m\ge2$.
Let $p$ be a prime divisor of~$m$.
Lemma~\ref{lemma: covering divisors} implies that $S$ is a covering set of~$\calT_p$.

On the one hand there exists~$s_1\in S$
such that $\pi_p(s_1)=\pi_p(O)$.
Therefore, $p$ divides $r$. 

On the other hand there exists~$s_2\in S$
such that $\pi_p(s_2)\ne\pi_p(O)$.
Therefore, $p$ does not divide $\il(s_2,O)=r$.

This is a contradiction.
Hence $S$ is tori-transparent.
\qed

\vspace{2mm}

\begin{remark}
The finiteness of the set $S$ is crucial in Theorem~\ref{theorem: finite set and torus}.
For instance, the set
\[S=\{0,6\}\times\z\]
is an example of an infinite set, for which Theorem~\ref{theorem: finite set and torus} does not hold. 

\vspace{2mm}


Indeed, for every~$m$, the set~$S$ is not a covering set of~$\calT_m$ as $[1,0]_m\not\in\pi_m(S)$ for $m\ne5$ and $[2,0]_m\not\in\pi_m(S)$ for~$m=5$. 
%

\vspace{2mm}

Assume that there exists a circle through all points of~$S$
with centre~$(x,y)$.
The point~$(x,y)$ is at integer distance one from all points of~$\{0\}\times\z$, hence $\gcd(x,y-n) = 1$ for all $n \in \z$ and therefore $x = \pm 1$. Similarly, $(x,y)$ is at integer distance one from all points of~$\{6\}\times\z$, hence $\gcd(x-6,y-n) = 1$ for all $n \in \z$ and therefore $x-6 =\pm 1$.
We arrive at a contradiction.
\end{remark}

Finally let us say a few words
about the $\aff(2,\z)$-invariance of the property of being a covering set of a torus.

\begin{proposition}
Let $S$ be an integer set and $m$ an integer number.
The property of~$S$ to be a covering set of~$\calT_m$
is preserved under $\aff(2,\z)$.
\end{proposition}

\begin{proof}
Any element of~$\aff(2,\z)$ can be written as a map $v\mapsto Av+b$
for some matrix~$A\in\GL(2,\z)$ and vector~$b\in\z^2$.
Note that the equation $v_1\equiv v_2 \modd m$ (coordinate-wise)
is equivalent to the equation $Av_1+b\equiv Av_2+b\modd m$.
So the number of points in the image under the projection~$\pi_m$ is preserved under~$\aff(2,\z)$.
\end{proof}

\begin{corollary}\label{aff invariance torus}
The property of a finite set to be tori-transparent  is invariant under $\aff(2,\z)$.
\qed
\end{corollary}

\begin{definition}
\label{def-S/r}
Let $S$ be an integer set and $k$ a positive integer.
We say that $S$ is {\it shift-divisible} by~$k$
if there exists an integer point~$x$ and an integer set~$\hat S$ such that \[S = x + k\hat S.\]
We then say that $\hat S\approx S/k$.
Note that $S$ is shift-divisible by~$k$
if and only if any two points in~$S$ are equivalent modulo~$k$.
Note that the set~$\hat S$ is uniquely defined up to a translation by an integer vector.
We define $S/k$ as the equivalence class of~$\hat S$ under translations by integer vectors.
The property of an integer set to be a covering set of~$\calT_m$ is preserved under translations by integer vectors,
hence we can say that $S/k$ {\it is a covering set of}~$\calT_m$ or is {\it tori-transparent\/} if the set~$\hat S$ has this property.
\end{definition}

\begin{proposition}
\label{prop-shift-div-mult}
Let $S$ be a finite integer set and $a,b$ integers.
If $S$ is shift-divisible by~$a$ and~$b$ then $S$ is shift-divisible by~$\lcm(a,b)$.
\end{proposition}

\begin{proof}
If $S$ is shift-divisible by~$a$ and~$b$
then any two points in~$S$ are equivalent modulo~$a$ and modulo~$b$
and therefore equivalent modulo~$\lcm(a,b)$.
Hence $S$ is shift-divisible by~$\lcm(a,b)$.
\end{proof}

\begin{proposition}
\label{prop-radius-r}
Let $S$ be a finite integer set and $r$ an integer.
Then $S$ has a circumscribed circle of radius~$r$
if and only if $S$ is shift-divisible by~$r$ and $S/r$ is tori-transparent.
\end{proposition}

\begin{proof}
Suppose that the set~$S$ has a circumscribed circle~$C$ of radius~$r$ with centre~$x$.
Then $S-x\subset r\z^2$ and $\hat S=(S-x)/r$ is an integer set such that $S=x+r \hat S$,
i.e.\ $S$ is shift-divisible by~$r$ and $S/r\approx\hat S$.
Moreover, $\hat C=(C-x)/r$ is a unit integer circumscribed circle of~$\hat S$,
hence $1\in\LaZ(\hat S)$.
Theorem~\ref{theorem: finite set and torus} implies that $S/r$ is tori-transparent.

Now suppose that $S$ is shift-divisible by~$r$ and $S/r$ is tori-transparent, i.e.\
there exists an integer point~$x$ and an integer tori-transparent set~$\hat S$ such that~$S=x+r \hat S$.
By Theorem~\ref{theorem: finite set and torus},
the set~$\hat S$ admits a unit integer circumscribed circle~$\hat C$.
Then $C=x+r \hat C$ is an integer circumscribed circle of~$S$ of radius~$r$.
\end{proof}

\begin{proposition}
\label{prop-LaZ-mult}
Let $S$ be a finite integer set and $a,b$ integers.
If $a,b\in\LaZ(S)$ then $\lcm(a,b)\in\LaZ(S)$.
\end{proposition}

\begin{proof}
If $a,b\in\LaZ(S)$ then Proposition~\ref{prop-radius-r} implies
that $S$ is shift-divisible by~$a$ and~$b$
and $S/a$, $S/b$ are tori-transparent.
Proposition~\ref{prop-shift-div-mult} implies that $S$ is shift-divisible by~$\lcm(a,b)$.
Let $\hat S=S/(\lcm(a,b))$.
Let $d=\gcd(a,b)$, $\hat a=a/d$ and $\hat b=b/d$,
so that $\gcd(\hat a,\hat b)=1$ and $\lcm(a,b)=d\hat a\hat b$.
The set
\[\hat a\hat S=\hat a(S/(d\hat a\hat b))=S/(d\hat b)=S/b\]
is tori-transparent, hence $\hat S$ is not a covering set of~$\calT_m$
for all~$m$ co-prime with~$\hat a$.
Similarly, the set
\[\hat b\hat S=\hat b(S/(d\hat a\hat b))=S/(d\hat a)=S/a\]
is tori-transparent, hence $\hat S$ is not a covering set of~$\calT_m$
for all~$m$ co-prime with~$\hat b$.
The integers~$\hat a$ and~$\hat b$ are co-prime,
hence every integer~$m$ is co-prime with at least one of~$\hat a$ and~$\hat b$.
Therefore $\hat S$ is not a covering set of any~$\calT_m$ for~$m\ge2$,
i.e.~$\hat S=S/(\lcm(a,b))$ is tori-transparent.
Proposition~\ref{prop-radius-r} implies that $\lcm(a,b)\in\LaZ(S)$.

\end{proof}

\section{Rational Circumscribed Circles}
\label{Rational Circumscribed Circles}

Some sets do not have integer circumscribed circles. However we can extend the definition of integer circumscribed circles to circles with rational radii.
We will see that every finite set has at least one rational circumscribed circle.

\subsection{Definition of a Rational Circumscribed Circle}

\begin{definition}
We call a fraction~$\frac{p}{q}$ {\it irreducible\/} if~$\gcd(p,q)=1$.
\end{definition}

\begin{definition}
\label{rational-circ-def}
Consider an integer set~$S$
and let $p$ and~$q$ be two integers.
We say that $S$ has a {\it rational circumscribed circle} of radius~$\frac{p}{q}$ if the set $qS$ has a circumscribed circle of radius~$p$.
\end{definition}

\begin{definition}
The {\it rational circumscribed spectrum}
$\LaQ(S)$ of an integer set~$S$ is the set of all rational values~$\frac{p}{q}$ such that $S$ admits a rational circumscribed circle of radius~$\frac{p}{q}$.
\end{definition}

\begin{remark}
Since every integer circle is also a rational circle, we have
$$\LaZ(S)\subset\LaQ(S).$$
\end{remark}

\begin{proposition}
Let $S$ be an integer set.
If $\frac{p}{q}$ is an irreducible fraction in~$\LaQ(S)$
and $a,b\in S$
then $p$ divides $\id(a,b)$.
\end{proposition}

\begin{proof}
By definition, $\frac{p}{q}\in\LaQ(S)$ implies $p\in\LaZ(q S)$,
i.e.\ the set~$qS$ has an integer circumscribed circle of radius~$p$.
Proposition~\ref{max-integer} implies that $p$ is a divisor of $\id(qa,qb)=q\cdot\id(a,b)$ for any~$a,b\in S$.
As $p$ and~$q$ are co-prime, it follows that $p$ is a divisor of~$\id(a,b)$.
\end{proof}

\noindent
This proposition implies that the rational spectrum is bounded.

\begin{corollary}\label{lamda q is boundeds}
Let $S$ be an integer set, $|S|\ge2$. 
Then the rational spectrum~$\LaQ(S)$ of~$S$
and the set of numerators of irreducible fractions in $\LaQ(S)$ are bounded.\qed 
\end{corollary}

\begin{proposition}
\label{prop-LaQ-mult}
Let $S$ be a finite integer set.
If $\frac{p_1}{q_1}$ and~$\frac{p_2}{q_2}$ are two irreducible fractions in~$\LaQ(S)$ then
\[\frac{\lcm(p_1q_2,p_2q_1)}{q_1q_2}=\frac{\lcm(p_1,p_2)}{\gcd(q_1,q_2)}\in\LaQ(S).\]
\end{proposition}

\begin{proof}
If $\frac{p_1}{q_1}\in\LaQ(S)$
then $p_1\in\LaZ(q_1 S)$,
hence $p_1q_2\in\LaZ(q_1q_2 S)$.
Similarly, $p_2q_1\in\LaZ(q_1q_2 S)$.
Proposition~\ref{prop-LaZ-mult} implies $\lcm(p_1q_2,p_2q_1)\in\LaZ(q_1q_2 S)$,
hence
\[\frac{\lcm(p_1q_2,p_2q_1)}{q_1q_2}\in\LaQ(S).\]
Finally, we will use the following identity known in elementary number theory
\[\frac{\lcm(p_1q_2,p_2q_1)}{q_1q_2}=\frac{\lcm(p_1,p_2)}{\gcd(q_1,q_2)}.\qedhere\]
\end{proof}

\subsection{Structure of Rational Spectra}
\label{subsection-structure}

\begin{proposition}
\label{prop-max-min-LaQ}
Let $S$ be a finite integer set.
If $\frac{p}{q}$ and~$\frac{p'}{q'}$ are two irreducible fractions in~$\LaQ(S)$
and $\max(\LaQ(S))=\frac{p}{q}$,
then $p'\,|\,p$ and~$q\,|\,q'$.
\end{proposition}

\begin{proof}
By Proposition~\ref{prop-LaQ-mult},
the number
\[\frac{\lcm(p,p')}{\gcd(q,q')}\]
is in $\LaQ(S)$, hence
\[\frac{\lcm(p,p')}{\gcd(q,q')}\le\max(\LaQ(S))=\frac{p}{q}.\]
Note that $\lcm(p,p')\ge p$ and $\gcd(q,q')\le q$, hence the inequality above can only hold if 
\[\lcm(p,p')=p,\qquad\gcd(q,q')=q.\]
Therefore $p'\,|\,p$ and $q\,|\,q'$.
\end{proof}

\begin{corollary}
Let $S$ be a finite integer set.
If $\frac{p}{q}$ is an irreducible fraction in~$\LaQ(S)$
and~$\max(\LaQ(S))=\frac{p}{q}$,
then $p$ is the largest possible numerator 
and $q$ is the smallest possible denominator
of an irreducible fraction in~$\LaQ(S)$.
\end{corollary}

\begin{theorem}
\label{theorem:rational circumscribed}
Let $S$ be a finite integer set.
Let $\{t_1,\ldots,t_n\}$ be the set of all primes~$t$ such that $S$ is a covering set of~$\calT_t$.
Let $\tau=\prod\limits_{i=1}^n t_i$.
Then there exists $p\in\z_+$ such that 
\[\LaQ(S)=\left\{\frac{1}{c}\cdot\frac{p}{\tau}\,\bigg|\,c\in\z_{+}\right\}.\]
In fact, $p=\max(\LaZ(\tau S))$, $p/\tau =\max(\LaQ(S))$,
and the greatest common divisor of all integer distances between pairs of points in~$S$ is a multiple of~$p$.

If $S$ is tori-transparent then $\tau=1$,
\[\LaQ(S)=\left\{\frac{p}{c}\,\bigg|\,c\in\z_{+}\right\}\]
and  $p=\max(\LaZ(S))=\max(\LaQ(S))$.
\end{theorem}

\begin{proof}
Let $\frac{p}{q}$ be an irreducible fraction such that $\max(\LaQ(S))=\frac{p}{q}$.
\begin{enumerate}
\item
We will show that $q$ is a divisor of~$\tau$:
We know that $S$ and hence $\tau S$ is not a covering set of~$\calT_t$
for any prime~$t\not\in\{t_1,\ldots,t_n\}$.
For $i=1,\dots,n$, the set $t_i S$ and hence $\tau S$ is not a covering set of~$\calT_{t_i}$.
In summary, the set~$\tau S$ is not a covering set of~$\calT_t$ for every prime~$t$, i.e.\ $\tau S$ is tori-transparent.
Theorem~\ref{theorem: finite set and torus} implies $1\in\LaZ(\tau S)$ and hence $\frac{1}{\tau}\in\LaQ(S)$.
Proposition~\ref{prop-max-min-LaQ} implies $q\,|\,\tau$.

\item
We will now show that $q=\tau$:
We have shown that $q$ is a divisor of~$\tau$.
Suppose that $q\ne\tau$ then $q$ is the product of some but not all of~$t_1,\dots,t_n$.
We can assume without loss of generality that $t_1$ is not a divisor of~$q$.
We know that $S$ is a covering set of~$\calT_{t_1}$ and $\gcd(t_1,q)=1$,
hence $q S$ is also a covering set of~$\calT_{t_1}$
and therefore not tori-transparent.
Theorem~\ref{theorem: finite set and torus} implies that $\LaZ(qS)=\emptyset$.
On the other hand, we know that $\frac{p}{q}\in\LaQ(s)$, hence $p\in\LaZ(q S)$ in contradiction to~$\LaZ(q S)=\emptyset$.
Hence $q=\tau$.

\item
We will next show that
\[\LaQ(S)\subset\left\{\frac{1}{c}\cdot\frac{p}{\tau}\,\bigg|\,c\in\z_{+}\right\}:\]
Consider an irreducible fraction~$\frac{p'}{q'}$ in~$\LaQ(S)$.
We know that the irreducible fraction~$\frac{p}{q}=\frac{p}{\tau}$
is the maximum of~$\LaQ(S)$.
Proposition~\ref{prop-max-min-LaQ} implies that~$p'\,|\,p$ and~ $\tau\,|\,q'$,
hence there exists~$c\in\z_{+}$ such that
\[\frac{p'}{q'}=\frac{1}{c}\cdot\frac{p}{\tau}.\]

\item
We will now show that
\[\left\{\frac{1}{c}\cdot\frac{p}{\tau}\,\bigg|\,c\in\z_{+}\right\}\subset\LaQ(S):\]
Let $c\in\z_+$.
We know that $\frac{p}{\tau}\in\LaQ(S)$,
hence $p\in\LaZ(\tau S)$ and therefore $\LaZ(\tau S)\ne\emptyset$.
Theorem~\ref{theorem: finite set and torus} implies
that the set~$\tau S$ is tori-transparent.
It follows that the set~$c(\tau S)$ is also tori-transparent.
Theorem~\ref{theorem: finite set and torus} implies that $1\in\LaQ(c\tau S)$ and therefore $\frac{1}{c}\in\LaQ(\tau S)$.
We know that $p,\frac{1}{c}\in\LaQ(\tau S)$, hence $\frac{p}{c}\in\LaQ(\tau S)$ according to Proposition~\ref{prop-LaQ-mult}.
Therefore $\frac{p}{c\tau}\in\LaQ(S)$.

\item
Finally, we will show that the greatest common divisor of all integer distances between pairs of points in~$S$ is a multiple of~$p$:
We know that $\frac{p}{q}=\frac{p}{\tau}\in\LaQ(S)$,
hence $p\in\LaZ(\tau S)$ and therefore $\tau S$ has a circumscribed circle of radius~$p$.
It follows that the integer distance between any two points in~$\tau S$ is a multiple of~$p$.
We know that $\gcd(p,\tau)=\gcd(p,q)=1$, hence the integer distance between any two points in~$S$ is a multiple of~$p$.\qedhere
\end{enumerate}
\end{proof}

\begin{remark}
Let $S$ be a finite integer set.
Then
$$\LaZ(S)=\LaQ(S)\cap\z.$$
In the case $\LaZ(S)\ne\emptyset$,
we additionally get the equality
$$\max(\LaZ(S))=\max(\LaQ(S)).$$
\end{remark}

\begin{remark}
There is a similarity between the expression for the rational circumscribed spectrum in Theorem~\ref{theorem:rational circumscribed}
and some formulas for coefficients of Ehrhart polynomials,
see for example~\cite{robins2015}. 
\end{remark}

\vspace{1mm}

\noindent
Note that while Theorem~\ref{theorem:rational circumscribed} states that the greatest common divisor of all integer distances between pairs of points in~$S$ is a multiple of~$p$, it is not necessarily equal to~$p$ as can be seen in the following example: 

\begin{example}
Consider the set
\[S=\{(0,0),(2,0),(0,2),(2,2)\}.\]
On the one hand, the set~$S$ is tori-transparent, so Theorem~\ref{theorem:rational circumscribed} implies that
there exists a divisor~$p$ of all integer distances between pairs of points in~$S$ such that
\[\LaQ(S)=\left\{\frac{p}{c}\ \bigg|\ c\in\z_{+}\right\}.\]
The greatest common divisor of all integer distances between points in~$S$
is~$g=2$,
hence either $p=1$ and $\LaZ(S)=\{1\}$
or~$p=g=2$ and $\LaZ(S)=\{1,2\}$.
On the other hand, we have $S=2\hat S$, where
\[\hat S=\{(0,0),(1,0),(0,1),(1,1)\}.\]
Now $\hat S$ is a covering set of~$\calT_2$, hence 
Theorem~\ref{theorem: finite set and torus} implies $1\not\in\LaZ(\hat S)$
and therefore $2\not\in\LaZ(S)$.
Thus~$p=1\ne g$.
\end{example}

\noindent
To give a more precise description of circumscribed spectra, we will need the following definition:

\begin{definition}
An integer set~$S$ is called {\it primitive}
if it is not shift-divisible by~$k$ for any integer~$k>1$.
\end{definition}

\begin{remark}
Note that a set is primitive if and only if the greatest common divisor of the distances between all pairs of its points equals to one.
\end{remark}

\begin{theorem}
\label{theorem:rational circumscribed precise}
Let $S$ be a finite integer set.
Let $x$ be an integer point, $g$ an integer and $\hat S$ a primitive set such that $S=x+g \hat S$.
Let $\{t_1,\ldots,t_n\}$ be the set of all primes~$t$ such that $\hat S$ is a covering set of~$\calT_t$.
Let $\tau=\prod\limits_{i=1}^n t_i$.
Then the rational circumscribed spectrum of~$S$ is
\[\LaQ(S)=\left\{\frac{1}{c}\cdot\frac{g}{\tau}\ \bigg|\ c\in\z_{+}\right\}.\]
If $\hat S$ is tori-transparent then $\tau=1$ and
\[\LaQ(S)=\left\{\frac{g}{c}\ \bigg|\ c\in\z_{+}\right\}.\]
\end{theorem}

\begin{proof}
Theorem~\ref{theorem:rational circumscribed} implies that there exists $p\in\z_+$ such that
\[\LaQ(\hat S)=\left\{\frac{1}{c}\cdot\frac{p}{\tau}\,\bigg|\,c\in\z_{+}\right\},\]
and that $p$ is a divisor of all integer distances between pairs of points in~$\hat S$.
The set~$\hat S$ is primitive, hence the greatest common divisor of all integer distances between pairs of points in~$\hat S$ is equal to~$1$ and therefore $p=1$.
It follows that 
\[
  \LaQ(\hat S)
  =\left\{\frac{1}{c}\cdot\frac{1}{\tau}\,\bigg|\,c\in\z_{+}\right\}
\]
and therefore
\[
  \LaQ(S)
  =\LaQ(g\hat S)
  =g\cdot\LaQ(\hat S)=\left\{\frac{1}{c}\cdot\frac{g}{\tau}\,\bigg|\,c\in\z_{+}\right\}.\qedhere
\]
\end{proof}


\begin{definition}
The \textit{primorial}~$d\#$ of~$d\in\z_+$ is defined
as the product of all prime numbers smaller or equal to~$d$.
\end{definition}

\begin{proposition}
Let $S$ be a finite integer set and $k=|S|$
then
\[\frac1{\lfloor\sqrt{k}\rfloor\#}\in\LaQ(S).\]
\end{proposition}

\begin{proof}
Let $k=|S|$.
Theorem~\ref{theorem:rational circumscribed} implies
that $\frac{1}{n\tau}\in\LaQ(S)$ for every~$n\in\z_+$, 
where $\tau=\prod\limits_{i=1}^n t_i$
and $\{t_1,\ldots,t_n\}$ is the set of all primes~$t$ such that $S$ is a covering set of~$\calT_t$.
Note that if $S$ is a covering set of an integer torus~$\calT_t$
then $t^2=|\calT_t|\le|S|=k$ and therefore $t\le\sqrt{k}$.
It follows that $\{t_1,\ldots,t_n\}$ is a subset of the set of all primes smaller or equal to~$\sqrt{k}$,
hence $\tau$ is a divisor of $\lfloor\sqrt k\rfloor\#$,
i.e.\ $\lfloor\sqrt k\rfloor\#=n\tau$ for some~$n\in\z_+$.
Therefore
\[\frac1{\lfloor\sqrt{k}\rfloor\#}=\frac{1}{n\tau}\in\LaQ(S).\qedhere\]
\end{proof}

\begin{example}
Let~$a,b\ge2$ be integers.
The circumscribed spectra of the integer set
$$G_{a,b}=\{1,\ldots,a\}\times\{1,\ldots,b\}$$
are given by
\[
  \LaZ(G_{a,b})=\emptyset,\quad
  \LaQ(G_{a,b})=\left.\left\{\frac{1}{c}\cdot\frac{1}{(\min(a,b))\#}\,\right|\,c\in\z_+\right\}.
\]

\vspace{1mm}

\noindent
To prove this, note that $G_{a,b}$ is a primitive set.
Theorem~\ref{theorem:rational circumscribed precise} implies that
\[\LaQ(G_{a,b})=\left\{\frac{1}{c}\cdot\frac{1}{\tau}\ \bigg|\ c\in\z_{+}\right\},\]
where $\{t_1,\ldots,t_n\}$ is the set of all primes~$t$ such that $G_{a,b}$ is a covering set of~$\calT_t$
and $\tau=\prod\limits_{i=1}^n t_i$.
The set~$G_{a,b}$ is a covering set for an integer torus~$\calT_t$ if and only if $2\le t\le\min(a,b)$.
Hence the set~$\{t_1,\dots,t_n\}$ consists of all primes smaller or equal to~$\min(a,b)$ and therefore $\tau=(\min(a,b))\#$.
Finally, $\LaZ(S)=\LaQ(S)\cap\z=\emptyset$.
\end{example}

\section{Circumscribed Circles of Polygons}
\label{Section-polygons}

We define an {\it integer circumscribed circle of a polygon}~$P$ as the integer circumscribed circle of the set of vertices of~$P$ in the sense of Definition~\ref{circ-def-1}.
Note that an integer circle is an integer circumscribed circle of~$P$ if and only if all vertices of~$P$ are on the circle (see Figure~\ref{fig:starburst}).
We define a {\it rational circumscribed circle of a polygon}~$P$ as the rational circumscribed circle of the set of vertices of~$P$ in the sense of Definition~\ref{rational-circ-def}.

\vspace{1mm}

In this section we summarise the implications of the results of Theorem~\ref{theorem:rational circumscribed precise} for integer and rational circumscribed circles of polygons.

\subsection{Circumscribed Circles of Segments and Triangles}

An integer segment or triangle always admits a unit circumscribed integer circle.

\begin{proposition}
\label{segment-triangle-circumscribed}
Let $S$ be an integer segment or triangle.
Let $g$ be the greatest common divisor of all integer distances between pairs of vertices of~$S$.
Then the integer circumscribed spectrum~$\LaZ(S)$ consists of all positive divisors of~$g$ and
\[\LaQ(S)=\left\{\left.\frac{g}{c}\,\right|\, c\in\z_+\right\}.\]
In particular if $S$ is a primitive segment or triangle then
\[
  \LaZ(S)=\{1\},
  \quad 
  \LaQ(S)=\left\{\left.\frac{1}{c}\,\right|\,c\in\z_+\right\}.
\]
\end{proposition}

\begin{proof}
There exist an integer point~$x$ and a primitive polygon~$\hat S$ such that $S=x+g\hat S$.
The set of vertices of~$\hat S$ consists of at most three points and therefore is tori-transparent.
Theorem~\ref{theorem:rational circumscribed precise} implies that
\[\LaQ(S)=\left\{\frac{g}{c}\ \bigg|\ c\in\z_{+}\right\}.\]
It follows that $\LaZ(S)=\LaQ(S)\cap\z$ consists of all positive divisors of~$g$.
\end{proof}

\noindent
We obtain the following corollary:

\begin{corollary}
If an integer set~$S$ has a circumscribed integer circle of radius~$r$ then the integer distance between any two points of~$S$ is a multiple of~$r$.
\end{corollary}

\begin{proof}
Consider $A,B\in S$.
Any integer circumscribed circle of $S$ is in particular an integer circumscribed circle of the segment~$AB$,
hence the integer length of the segment~$AB$ is divisible by $r$.
\end{proof}

Let us recall the Euclidean Extended Sine Rule:
for a triangle $ABC$  we have
$$
\frac{|AB|}{\sin \angle BCA }=
\frac{|BC|}{\sin \angle CAB }=
\frac{|CA|}{\sin \angle ABC }=
2R,
$$
where $R$ is the radius of the circumscribed circle.

\vspace{1mm}

\noindent
As was shown in~\cite{OK2008},
the first two of these equalities hold in lattice geometry:
$$
\frac{\il(AB)}{\isin \angle BCA}=
\frac{\il(BC)}{\isin \angle CAB}=
\frac{\il(CA)}{\isin \angle ABC}.
$$
Proposition~\ref{segment-triangle-circumscribed} tells us that there is no natural generalisation for the last equality.
Indeed, the circumscribed spectrum depends entirely on the integer length of the edges of the triangle and does not depend on the angles.

\vspace{2mm}

\noindent
For instance consider two triangles, one with vertices $(0,0)$, $(1,0)$, $(0,1)$ and another with vertices $(0,0)$, $(1,2)$, $(2,1)$.
For both triangles, all edges are of unit integer length.
The sets of integer sines of the angles of these triangles are distinct, for the first triangle all integer sines are equal to~$1$ while for the second triangle all integer sines of the angles are equal to~$3$.
Nevertheless the circumscribed spectra for both triangles coincide.









    

\subsection{Circumscribed Circles of Quadrangles}
We have seen that every triangle has an integer circumscribed circle, however this is no longer true for quadrangles as the following example shows.

\vspace{2mm}

\begin{definition}
An integer polygon $P$ is {\it empty} if the only lattice points contained in $P$ are the vertices. 
\end{definition}

\begin{proposition}
\label{empty square}
An empty integer strictly convex quadrilateral
does not have a circumscribed integer circle.
\end{proposition}

\begin{proof}
Note that every empty integer strictly convex quadrilateral is integer congruent to the coordinate square $S_1$ with vertices $(0,0)$, $(1,0)$, $(1,1)$ and~$(0,1)$.
The square~$S_1$ is a covering set of~$\calT_2$,
hence it is not tori-transparent.
Theorem~\ref{theorem: finite set and torus} implies that $S_1$ does not admit integer circumscribed circles of any radius.
\end{proof}

\noindent
However some quadrangles have integer circumscribed circles.

\begin{example}
The quadrilateral with vertices $A=(0,0)$, $B=(1,0)$, $C=(0,1)$ and~$D=(2,2)$ has a unit circumscribed circle centred at $(1,1)$.
\end{example}

The situation is similar to the Euclidean geometry, where a quadrangle has a circumscribed circle if and only if its opposite angles add up to~$\pi$.
The lattice version of this rule is as follows:

\begin{proposition}
\label{prop:4gon-t2}
An integer quadrangle has an integer circumscribed circle if and only if the set of its vertices is not a covering set of $\calT_2$.    
\end{proposition}

\begin{proof}
Theorem~\ref{theorem: finite set and torus} implies that a quadrangle admits an integer circumscribed circle
if and only if its set of vertices~$V$ is tori-transparent,
i.e.\ is not a covering set of any integer torus~$\calT_t$ for $t\ge2$.
The set~$V$ cannot be a covering set of~$\calT_t$ for $t>2$ since $|V|=4<t^2=|\calT_t|$.
Hence the set~$V$ is tori-transparent if and only if it is not a covering set of~$\calT_2$.
\end{proof}

\begin{remark}
The conditions for a quadrangle to admit a circumscribed circle can be stated in terms of the  parity of the six integer distances between its pairs of vertices as follows:
An integer quadrangle admits an integer circumscribed circle if and only if at least one of the integer distances between its vertices is even.

\vspace{1mm}

On the other hand, the existence of an integer circumscribed circle is not determined solely by the integer angles of the integer quadrangle.
For example, the angles of the quadrangles with vertices $A(0,0)$, $B(0,1)$, $C(1,1)$, $D(1,0)$ and 
$P(-1,0)$, $Q(-1,1)$, $R(0,1)$, $S(1,0)$
are congruent to each other,
however the latter one admits a circumscribed circle, for example one centered at the origin $O(0,0)$, while the former one does not.

$$
\includegraphics[scale = 1.3]{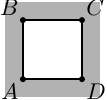}
\qquad\qquad\qquad\qquad
\includegraphics[scale = 1.3]{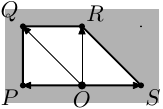}
$$

\end{remark}

\subsection{Circumscribed Circles of General Polygons}

In fact, the argument used in the proof of Proposition~\ref{prop:4gon-t2}
holds for all $n$-gons with~$n\le 8$:

\begin{proposition}\label{8gon}
An integer $n$-gon with~$n\le8$ has an integer circumscribed circle if and only if the set of its vertices is not a covering set of $\calT_2$.\qed    
\end{proposition}


In general, the following statement holds:

\begin{proposition}
An integer $n$-gon admits an integer circumscribed circle if and only if its vertices are not a covering set of~$\calT_t$ for every~$t\le\sqrt n$.
\qed
\end{proposition}

\noindent
{\bf Acknowledgments:}
RS was supported by an LMS Vacation Bursary.

\bibliographystyle{plain} 
\bibliography{ReferencesPreprint} 

\end{document}